\documentclass{article}
\usepackage{amssymb}
\usepackage{amscd} 
\usepackage{amsmath} 
\usepackage{amsthm}
\usepackage{color}
\usepackage{dsfont}
\usepackage{enumitem}
\usepackage{graphicx}
\usepackage[utf8]{inputenc} 
\usepackage{latexsym}
\usepackage{mathtools}
\usepackage{subcaption}
\usepackage{tikz}
\usepackage{mleftright}
\mleftright
\tikzstyle{vertex}=[circle, draw, fill, inner sep=0pt, minimum width=4pt]

\newcommand*\samethanks[1][\value{footnote}]{\footnotemark[#1]}

\newtheorem{theorem}{Theorem}[section]
\newtheorem{lemma}[theorem]{Lemma} 
\newtheorem{corollary}[theorem]{Corollary}

\newtheorem{problem}[theorem]{Problem}

\theoremstyle{definition}

\numberwithin{equation}{section}

\newcommand{\ceil}[1]{\left\lceil #1 \right\rceil}
\newcommand{\expec}[1]{\mathbb{E} \left[ #1 \right]}
\DeclareMathOperator{\Bin}{Bin}
\DeclareMathOperator{\Poi}{Poisson}
\newcommand{\bfa}{\mathbf{a}}
\newcommand{\bfb}{\mathbf{b}}

\title{Decomposing random permutations into order-isomorphic subpermutations}
\author{Carla Groenland\thanks{Utrecht University, Utrecht, The Netherlands. Partially supported by CRACKNP with funding from the European Research Council under the EU Horizon 2020 research and innovation programme (grant agreement no. 853234).}\and
Tom Johnston\thanks{School of Mathematics, University of Bristol, Bristol, BS8 1UG, UK and Heilbronn Institute for Mathematical Research, Bristol, UK.}\and 
D\'aniel Kor\'andi\thanks{Mathematical Institute, University of Oxford, Oxford, OX2 6GG, UK.} \thanks{Supported by SNSF Postdoc.Mobility Fellowship P400P2\_186686}\and
Alexander Roberts\samethanks[3]\and
Alex Scott\samethanks[3] \thanks{Supported by EPSRC grant EP/V007327/1.}\and
Jane Tan\samethanks[3]}

\date{}

\begin{document}

\maketitle
\begin{abstract}
	Two permutations $\sigma$ and $\pi$ are $\ell$-similar if they can be decomposed into subpermutations $\sigma^{(1)}, \ldots, \sigma^{(\ell)}$ and $\pi^{(1)}, \ldots, \pi^{(\ell)}$ such that $\sigma^{(i)}$ is order-isomorphic to $\pi^{(i)}$ for all $i \in [\ell]$. Recently, Dudek, Grytczuk and Ruci{\'n}ski posed the problem of determining the minimum $\ell$ for which two permutations chosen independently and uniformly at random are $\ell$-similar. We show that two such permutations are $O\big(n^{1/3}\log^{11/6}(n)\big)$-similar with high probability, which is tight up to a polylogarithmic factor. Our result also generalises to simultaneous decompositions of multiple permutations.
\end{abstract}

\section{Introduction}

Given a sequence $\bfa = a_1, \dots, a_{n}$ of $n$ distinct real numbers, one can uniquely associate a permutation $\sigma \in S_n$ (which we write as a sequence $\sigma(1),\sigma(2),\ldots,\sigma(n)$ with elements in $[n]$) to $\bfa$ by replacing the $i$th smallest element of $\bfa$ with $i$. We call this $\sigma$ the \emph{pattern} of $\bfa$. Two sequences $\bfa = a_1, \dots, a_{n}$ and $\bfb = b_1, \dots, b_{n}$ of distinct real numbers are \emph{order-isomorphic} if they have the same pattern, and in this case we write $\bfa \sim \bfb$.

In this paper, we are concerned with the problem of decomposing two random permutations into order-isomorphic subpermutations and, in particular, how many parts are necessary. Formally, we say two permutations $\sigma$ and $\pi$ are \emph{$\ell$-similar} if the sequences $\sigma(1), \dots, \sigma(n)$ and $\pi(1), \dots, \pi(n)$ can be partitioned into subsequences $\sigma^{(1)}, \dots, \sigma^{(\ell)}$ and $\pi^{(1)}, \dots, \pi^{(\ell)}$ such that $\sigma^{(i)} \sim \pi^{(i)}$ for each $i\in[\ell]$, and we write $U(\sigma, \pi)$ for the smallest $\ell$ for which $\sigma$ and $\pi$ are $\ell$-similar. For example, if we take the permutations $\sigma$ and $\pi$ to be $1, 4, 3, 5, 2$ and $2, 5, 3, 1, 4$ respectively, then $\sigma$ and $\pi$ are $2$-similar. Indeed, we may decompose $\sigma$ into $\sigma^{(1)} = 1, 3, 2$ and $\sigma^{(2)} = 4,5$ and decompose $\pi$ into $\pi^{(1)} = 2, 5, 3$ and $\pi^{(2)} = 1, 4$. It is easy to check that $\sigma^{(1)} \sim \pi^{(1)}$ and $\sigma^{(2)} \sim \pi^{(2)}$, and so $\sigma$ and $\pi$ are $2$-similar. Since $\sigma$ and $\pi$ do not have the same pattern, they are not $1$-similar, and $U(\sigma, \pi) = 2$.

Similar decomposition problems have been studied in other contexts. For example, Ulam posed the problem of determining the minimum $\ell(n)$ such that any two graphs $G$ and $H$ on $n$ vertices can be decomposed into $\ell$ edge-disjoint subgraphs $G_1, \dots, G_\ell$ and $H_1, \dots, H_\ell$ such that $G_i \sim H_i$ for all $i \in [\ell]$. This led to a series of papers which culminated in the result that all graphs on $n$ vertices with the same number of edges may be simultaneously decomposed into $3n/4 + O(1)$ parts~\cite{chung1979minimal,chung1981minimal,chung1984minimal}.

Unfortunately, such a result cannot possibly hold for permutations; in the worst case two permutations might only be $n$-similar. Indeed, the identity permutation $1, \dots, n$ and the reverse permutation $n, \dots, 1$ contain no non-trivial order-isomorphic subpermutations, and so cannot be $(n-1)$-similar. However, it is far from obvious what happens on average.

\begin{problem}[Dudek, Grytczuk and Ruci{\'n}ski \cite{dudek2021variations}]\label{prob:decomp}
Let $\sigma$ and $\pi$ be two permutations of length $n$ chosen independently and uniformly at random. What is the expected value of $U(\sigma, \pi)$?
\end{problem}

A simple $O(\sqrt{n})$ upper bound can be proved by decomposing the permutations into increasing subpermutations as follows. The length of a longest increasing sequence in a uniformly random permutation is asymptotically $2 \sqrt{n}$ (see \cite{steele1995variations}), and this is well-concentrated \cite{bollobas1992height, frieze1991length}. In particular, if $L$ is the length of a longest increasing subpermutation in a uniformly random permutation, then
\[\mathbb{P}\left(\left|L - \expec{L}\right| \geq n^{\frac{1}{3}}\right) \leq \exp(-n^{\beta})\]
for some $\beta > 0$. Hence, rather crudely, we have $L \leq 3 \sqrt{n}$ with very high probability. By Dilworth's Theorem, this implies that a uniformly random permutation can be decomposed into at most $3 \sqrt{n}$ decreasing subpermutations with very high probability. It follows easily that, with very high probability, a pair of uniformly random permutations $\sigma$ and $\pi$ can be decomposed into $\ell \leq 6 \sqrt{n}$ decreasing subsequences $\sigma^{(1)}, \dots, \sigma^{(\ell)}$ and $\pi^{(1)}, \dots, \pi^{(\ell)}$ such that $\sigma^{(i)}$ and $\pi^{(i)}$ have the same length for all $i \in [\ell]$.

Our main result improves the preceding simple bound to the following.
\begin{theorem}
	\label{thm:joint-twins}
	Let $\sigma$ and $\pi$ be two permutations of length $n$ chosen independently and uniformly at random. Then $\sigma$ and $\pi$ are $O\big(n^{1/3} \log^{11/6}(n)\big)$-similar with probability at least $1 - o(n^{-1})$.
\end{theorem}

This immediately provides the following bound on $\expec{U(\sigma, \pi)}$.

\begin{corollary}
	\label{cor:joint-twins}
	Let $\sigma$ and $\pi$ be two independent uniformly random permutations of length $n$. Then
	\[\expec{U(\sigma, \pi)} = O\left( n^{{1}/{3}} \log^{11/6}(n) \right).\]
\end{corollary}

The bound in Theorem~\ref{thm:joint-twins} (and Corollary~\ref{cor:joint-twins}) is tight up to the $\log^{11/6}(n)$ factor. Before seeing this, we make a brief detour to survey some closely related work on twinned subpermutations of a single permutation, from which our simple lower bound emerges. Formally, \emph{twins} $T = (T_1, T_2)$ of length $\ell$ in $\sigma \in S_n$ are a pair of disjoint subsequences $T_1 = \sigma_{i_1}, \dots, \sigma_{i_\ell}$ and $T_2 = \sigma_{j_1}, \dots, \sigma_{j_\ell}$ such that $T_1 \sim T_2$. The problem of finding long twins was introduced by Gawron \cite{gawron2014izomorficzne}, who showed that the longest twins guaranteed in a permutation of length $n$ is at most $O(n^{2/3})$ and conjectured that this should be tight (up to the constant). As observed in~\cite{gawron2014izomorficzne}, the Erd\H{o}s-Szekeres theorem provides a lower bound of $\Omega(n^{1/2})$. This lower bound was improved to $\Omega(n^{3/5})$ by Bukh and Rudenko \cite{bukh2020order}.

The proof of Gawron's upper bound comes from applying the first moment method to a uniformly random permutation, and there has since been work on finding a matching lower bound for random permutations. A lower bound of $\Omega\big(n^{2/3}/\log^{1/3}(n)\big)$ was shown by Dudek, Grytczuk and Ruci{\'n}ski~\cite{dudek2021variations}, and this was improved to the sharp bound of $\Omega(n^{2/3})$ by Bukh and Rudenko~\cite{bukh2020order}. There has also been further work of Dudek, Grytczuk and Ruci{\'n}ski in the same vein concerning \emph{$k$-twins}~\cite{dudek2021multiple}, which are a collection of $k$ pairwise disjoint order-isomorphic subpermutations of a single permutation, and other notions of twins with additional restrictions imposed or with a weaker similarity condition (see~\cite{dudek2022weak,dudek2021tight}).

Returning to Problem~\ref{prob:decomp}, if one chooses a permutation $\sigma$ of length $2n$ uniformly at random, then the patterns of the first and the second half of $\sigma$ are independent, uniformly random permutations of length $n$. Order-isomorphic subpermutations of length $\ell$ between the first and second of half of $\sigma$ give rise to twins of length $\ell$ in $\sigma$. Hence, from Gawron's result that the longest twin in a permutation of length $2n$ is of length $O(n^{2/3})$ with high probability~\cite{gawron2014izomorficzne}, it follows that the longest order-isomorphic subpermutations of two uniformly random permutations of length $n$ are also of length $O(n^{2/3})$ with high probability. This means that $\Omega(n^{1/3})$ pieces are needed with high probability.

The definition of similarity easily generalises to multiple permutations. We say that a collection of $k$ permutations $\sigma_1, \ldots, \sigma_k$ are \emph{$\ell$-similar} if each permutation $\sigma_j$ can be split into $\ell$ subpermutations $\sigma_{j}^{(1)}, \dots, \sigma_{j}^{(\ell)}$ such that $\sigma_{j}^{(i)} \sim \sigma_{j'}^{(i)}$ for all choices of $i \in [\ell]$, and $j, j' \in [k]$. This naturally leads to the problem of determining, for each $k$, the minimum $\ell(n)$ such that a set of $k$ permutations of length $n$ chosen independently and uniformly at random are $\ell$-similar. There is again a simple upper bound of $O(\sqrt{n})$ which comes from decomposing the permutations into increasing subpermutations. Using the result of Dudek, Grytczuk and Ruci{\'n}ski~\cite{dudek2021multiple} which states that the longest $k$-twins in a random permutation of length $n$ have length at most $O(n^{k/(2k-1)})$ with high probability, we can deduce a lower bound of $\Omega(n^{(k-1)/(2k-1)})$ in the same manner as the $k=2$ case.

Our proof of Theorem~\ref{thm:joint-twins} extends to give a bound in the setting of multiple permutations, and it is again tight up to a polylogarithmic factor.

\begin{theorem}
	\label{thm:multi}
	For any fixed integer $k\geq 2$, let $\sigma_1, \dots, \sigma_k$ be permutations of length $n$ chosen independently and uniformly at random. Then $\sigma_1, \dots, \sigma_k$ are $O\left(n^{\frac{k-1}{2k-1}} \log^{\frac{3(k-1)}{2} + \frac{1}{2k-1}}(n) \right)$-similar with probability at least $1 - o(n^{-1})$.
\end{theorem}

We give proofs in the next section, and discuss some further questions in the final section.

\section{Proof of the main theorems}
\label{sec:twins}
We begin with an outline of the proof of Theorem~\ref{thm:joint-twins}; the minor modifications required to prove Theorem~\ref{thm:multi} are given at the end of this section.

In the following, the idea of constructing an auxiliary bipartite multigraph has previously been used by Dudek, Grytczuk and Ruci{\'n}ski~\cite{dudek2021variations} and later by Bukh and Rudenko~\cite{bukh2020order} to study twins in a single random permutation. The use of a Poisson process to sample random permutations also appears in \cite{bukh2020order}.

It is convenient to work with random permutations generated from $n$ uniformly distributed points in the unit square $[0,1]^2$. With the $n$ points ordered by their $x$-coordinate, the permutation can be read off as the relative ordering of the $y$-coordinates. This representation was used by Bukh and Rudenko~\cite{bukh2020order}. Given sets of $n$ ``red" points and $n$ ``blue" points which define a pair of random permutations, we choose a matching between the red points and the blue points and only consider decompositions into subpermutations where the matched points are corresponding entries in the decomposition.

A useful strategy to simplify viewing patterns of permutations generated in this way is to discretise the unit square by splitting it up into an $M \times M$ grid of smaller squares, and consider the squares in which the points fall. Then, given a subpermutation in which there is at most one point in each row and column, the pattern of the subpermutation is determined by the squares in which each point lies and we do not need to know their positions within the squares. If we wish to use this subpermutation in our decomposition, we need the pattern of this subpermutation to be the same for the red points as for the matched blue points. We only consider the simple case where all of the matched points are offset by the same amount, and so we label each point by the relative position (measured in grid squares) of its matched point. We will handle each label using separate subpermutations, making it necessary to avoid having too many different possible labels. We will achieve this by ensuring matched points are close together. The following result of Leighton and Shor allows us to control the maximum distance between matched points.

\begin{theorem}[Leighton, Shor \cite{leighton1989tight}]
	\label{thm:leighton-shor}
	Suppose $A$ and $B$ are two independent sets of $n$ points uniformly distributed in $[0,1]^2$. Then there is a function $\alpha=\alpha(n)=\Omega(\sqrt{\log n})$ and an absolute constant $C$ such that, with probability at least $1 - n^{-\alpha}$, there is a perfect matching between $A$ and $B$ for which the maximum distance between two matched points is at most $C \log^{3/4}(n)/\sqrt{n}$.
\end{theorem}

To complete the proof, we further partition the points corresponding to a specific label in such a way that each part contains at most one point from each row and column. This can be reduced to the problem of edge-colouring an auxiliary bipartite multigraph, and it suffices to bound the number of points in a given region. If we were to naively generate $n$ points uniformly at random, then the numbers of points in different regions would not be independent, even if the regions were disjoint. Thus, we generate points using a Poisson process with mean $2n$ and randomly delete excess points to obtain $n$ uniformly distributed points. The number of our $n$ points in a given region is bounded by the number of points of the Poisson process in that region, and this is independent for disjoint regions.

We will make use of the following well-known bound on the tail probabilities of the Poisson distribution.

\begin{lemma}[Theorem 5.4 in \cite{mitzenmacher2017probability}]
	\label{lem:poisson-tail}
	Let $X \sim \Poi\left( \lambda\right)$ and $x > \lambda$. Then
	\[ \mathbb{P} \left( X \geq x \right) \leq \frac {(e\lambda )^{x}e^{-\lambda }}{x^{x}}.\]
\end{lemma}

We are now ready to give the formal argument.

\begin{proof}[Proof of Theorem~\ref{thm:joint-twins}]

	Given $n$ points $\left\{(a_i, b_i) : i \in [n]\right\}$ in the unit square $[0,1]^2$ such that $a_1 <  a_2 < \dotsb < a_n$, we can define a permutation $\sigma$ by setting $\sigma(i) = j$, where $b_i$ is the $j$th smallest element of the set $\{b_1, \dots, b_n\}$. If the $n$ points are chosen uniformly at random in the unit square, this process generates a uniformly random permutation.

	Let $R$ be a 2D Poisson process with rate $2n$ and let $N_R$ denote the number of points of $R$ in $[0,1]^2$. Then $N_R \sim \Poi(2n)$ and
	\[ \mathbb{P}\left( N_R \leq n \right) \leq \exp \left( -  \frac{n^2}{4n} \right).\]
	Conditional on the number of points in $[0,1]^2$ being $m \geq n$, the $m$ points are distributed uniformly. After removing $m-n$ points uniformly at random, we have $n$ uniformly distributed points, which we call the \emph{red points}, and we can use this to get a uniformly random permutation $\sigma \in S_n$. Similarly, we can use an independent Poisson process $B$ to define a set of $n$ uniformly distributed \emph{blue points} and a corresponding permutation $\pi \in S_n$.

	We now match up the $n$ red points with the $n$ blue points using Theorem~\ref{thm:leighton-shor}.

	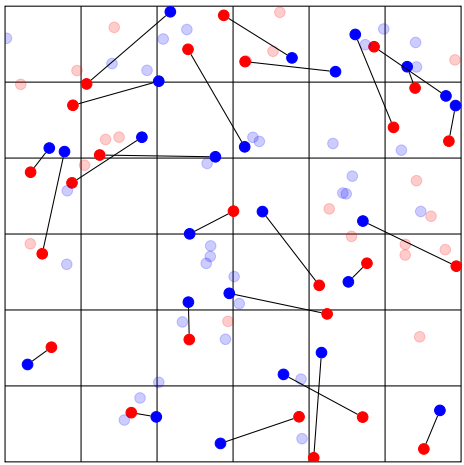
\begin{figure}
	\centering
	\begin{tikzpicture}[xscale=0.5\textwidth/1cm, yscale=0.5\textwidth/1cm]
		\draw (0,0) rectangle (1,1);
		\clip (0,0) rectangle (1,1);
		\draw (0, 0.16666666666666666) -- (1,0.16666666666666666);
		\draw (0.16666666666666666, 0) -- (0.16666666666666666,1);
		\draw (0, 0.3333333333333333) -- (1,0.3333333333333333);
		\draw (0.3333333333333333, 0) -- (0.3333333333333333,1);
		\draw (0, 0.5) -- (1,0.5);
		\draw (0.5, 0) -- (0.5,1);
		\draw (0, 0.6666666666666666) -- (1,0.6666666666666666);
		\draw (0.6666666666666666, 0) -- (0.6666666666666666,1);
		\draw (0, 0.8333333333333333) -- (1,0.8333333333333333);
		\draw (0.8333333333333333, 0) -- (0.8333333333333333,1);
		\draw node[style=vertex, red, opacity=1](r0) at (0.7062022986281737, 0.324695981345315) {};
		\draw node[style=vertex, red, opacity=1](r1) at (0.9183121824649292, 0.028324099172164646) {};
		\draw node[style=vertex, red, opacity=1](r2) at (0.27684008675089405, 0.10795370718547159) {};
		\draw node[style=vertex, red, opacity=1](r3) at (0.688964132031175, 0.3873405687215951) {};
		\draw node[style=vertex, red, opacity=1](r4) at (0.9732261563544754, 0.703766092413252) {};
		\draw node[style=vertex, red, opacity=1](r5) at (0.17882999162113236, 0.8293829848866756) {};
		\draw node[style=vertex, red, opacity=1](r6) at (0.5009156073330866, 0.5501677971073005) {};
		\draw node[style=vertex, red, opacity=1](r7) at (0.8091817311759989, 0.9110325833997828) {};
		\draw node[style=vertex, red, opacity=1](r8) at (0.10169909346112732, 0.25160783763136674) {};
		\draw node[style=vertex, red, opacity=1](r9) at (0.793569013681623, 0.4358683449015855) {};
		\draw node[style=vertex, red, opacity=1](r10) at (0.40121237516580055, 0.9053711225579613) {};
		\draw node[style=vertex, red, opacity=1](r11) at (0.8991532583697669, 0.8204098821475377) {};
		\draw node[style=vertex, red, opacity=1](r12) at (0.8518360621729151, 0.7340607988640941) {};
		\draw node[style=vertex, red, opacity=1](r13) at (0.05603357123186831, 0.6356293952095554) {};
		\draw node[style=vertex, red, opacity=1](r14) at (0.526397275588272, 0.8784053137975155) {};
		\draw node[style=vertex, red, opacity=1](r15) at (0.40401207249635, 0.26842294951046686) {};
		\draw node[style=vertex, red, opacity=1](r16) at (0.20755984390116453, 0.6733482699604347) {};
		\draw node[style=vertex, red, opacity=1](r17) at (0.7841996937064648, 0.09822530456504389) {};
		\draw node[style=vertex, red, opacity=1](r18) at (0.47958401300371645, 0.97988320932866) {};
		\draw node[style=vertex, red, opacity=1](r19) at (0.1468185772481496, 0.6122878788851661) {};
		\draw node[style=vertex, red, opacity=1](r20) at (0.644669289950765, 0.09893455882012035) {};
		\draw node[style=vertex, red, opacity=1](r21) at (0.9894998770087428, 0.4294308306054885) {};
		\draw node[style=vertex, red, opacity=1](r22) at (0.6768956574872981, 0.009005649517304318) {};
		\draw node[style=vertex, red, opacity=1](r23) at (0.14881834361344962, 0.782510551943364) {};
		\draw node[style=vertex, red, opacity=1](r24) at (0.08152244603214623, 0.4566352000150366) {};
		\draw node[style=vertex, red, opacity=0.2](r25) at (0.17433120711217823, 0.6510807806092181) {};
		\draw node[style=vertex, red, opacity=0.2](r26) at (0.5289390816957198, 0.8798298584184707) {};
		\draw node[style=vertex, red, opacity=0.2](r27) at (0.15775079672059336, 0.8586421207679388) {};
		\draw node[style=vertex, red, opacity=0.2](r28) at (0.9865372288181382, 0.8821464847718687) {};
		\draw node[style=vertex, red, opacity=0.2](r29) at (0.7597862992165564, 0.4948506485747609) {};
		\draw node[style=vertex, red, opacity=0.2](r30) at (0.6025969201650985, 0.9865535189579078) {};
		\draw node[style=vertex, red, opacity=0.2](r31) at (0.22040765839141138, 0.7074392775318027) {};
		\draw node[style=vertex, red, opacity=0.2](r32) at (0.965296762397882, 0.4657676517373167) {};
		\draw node[style=vertex, red, opacity=0.2](r33) at (0.8777488173319412, 0.47725831390813267) {};
		\draw node[style=vertex, red, opacity=0.2](r34) at (0.8774888094481618, 0.45397035895973303) {};
		\draw node[style=vertex, red, opacity=0.2](r35) at (0.034349189941555135, 0.8284577900022183) {};
		\draw node[style=vertex, red, opacity=0.2](r36) at (0.9340688915832758, 0.5390699024632541) {};
		\draw node[style=vertex, red, opacity=0.2](r37) at (0.9091673514255132, 0.27444516862345397) {};
		\draw node[style=vertex, red, opacity=0.2](r38) at (0.4887274789316395, 0.3083377983859898) {};
		\draw node[style=vertex, red, opacity=0.2](r39) at (0.24998265894745766, 0.7126599251188204) {};
		\draw node[style=vertex, red, opacity=0.2](r40) at (0.7108267061667395, 0.5552155264289989) {};
		\draw node[style=vertex, red, opacity=0.2](r41) at (0.9022014480282108, 0.6169614821608604) {};
		\draw node[style=vertex, red, opacity=0.2](r42) at (0.05542369353306321, 0.47861065206130454) {};
		\draw node[style=vertex, red, opacity=0.2](r43) at (0.587608327557369, 0.9011199850284789) {};
		\draw node[style=vertex, red, opacity=0.2](r44) at (0.23943532442467483, 0.9534894267860412) {};
		\draw node[style=vertex, blue, opacity=1](b0) at (0.6938976037330391, 0.23973900040897317) {};
		\draw node[style=vertex, blue, opacity=1](b1) at (0.3000736334880179, 0.7124350115982306) {};
		\draw node[style=vertex, blue, opacity=1](b2) at (0.7528561061327365, 0.394996879900452) {};
		\draw node[style=vertex, blue, opacity=1](b3) at (0.04936821924435675, 0.21369001954937808) {};
		\draw node[style=vertex, blue, opacity=1](b4) at (0.09705309579255642, 0.6886834849920827) {};
		\draw node[style=vertex, blue, opacity=1](b5) at (0.98787881984423, 0.7818451025162539) {};
		\draw node[style=vertex, blue, opacity=1](b6) at (0.6105225934515158, 0.19194519474705424) {};
		\draw node[style=vertex, blue, opacity=1](b7) at (0.36246567950048414, 0.9877394503704601) {};
		\draw node[style=vertex, blue, opacity=1](b8) at (0.4048108913656835, 0.5004020850407428) {};
		\draw node[style=vertex, blue, opacity=1](b9) at (0.5254765886937326, 0.6913273639945751) {};
		\draw node[style=vertex, blue, opacity=1](b10) at (0.9535537002018528, 0.11295314840676485) {};
		\draw node[style=vertex, blue, opacity=1](b11) at (0.9668115930107262, 0.8030551640855773) {};
		\draw node[style=vertex, blue, opacity=1](b12) at (0.33164126184037, 0.09869472771709766) {};
		\draw node[style=vertex, blue, opacity=1](b13) at (0.7846824204148839, 0.5284399200131894) {};
		\draw node[style=vertex, blue, opacity=1](b14) at (0.33704350716597264, 0.8353973852804982) {};
		\draw node[style=vertex, blue, opacity=1](b15) at (0.881749395978703, 0.8671308298783007) {};
		\draw node[style=vertex, blue, opacity=1](b16) at (0.6291805780292972, 0.8864767696647602) {};
		\draw node[style=vertex, blue, opacity=1](b17) at (0.4917101898423507, 0.36984602622062246) {};
		\draw node[style=vertex, blue, opacity=1](b18) at (0.4020456201628176, 0.35047670730054836) {};
		\draw node[style=vertex, blue, opacity=1](b19) at (0.5644461012252803, 0.5492333099529227) {};
		\draw node[style=vertex, blue, opacity=1](b20) at (0.7679801384914596, 0.9379348089869762) {};
		\draw node[style=vertex, blue, opacity=1](b21) at (0.7246061308308218, 0.8562918767366846) {};
		\draw node[style=vertex, blue, opacity=1](b22) at (0.13039169443112283, 0.6808427619943374) {};
		\draw node[style=vertex, blue, opacity=1](b23) at (0.46125777361941556, 0.6694881357773304) {};
		\draw node[style=vertex, blue, opacity=1](b24) at (0.4724058726077113, 0.04043879657051285) {};
		\draw node[style=vertex, blue, opacity=0.2](b25) at (0.6512103799338597, 0.050881493975450914) {};
		\draw node[style=vertex, blue, opacity=0.2](b26) at (0.761515854177336, 0.6270004109894689) {};
		\draw node[style=vertex, blue, opacity=0.2](b27) at (0.4832727651093192, 0.2690547559058465) {};
		\draw node[style=vertex, blue, opacity=0.2](b28) at (0.5434773808125037, 0.7122399274444818) {};
		\draw node[style=vertex, blue, opacity=0.2](b29) at (0.2959732318120147, 0.14024739059625807) {};
		\draw node[style=vertex, blue, opacity=0.2](b30) at (0.33710939469887313, 0.1745255130678759) {};
		\draw node[style=vertex, blue, opacity=0.2](b31) at (0.5135574564327725, 0.34734270526699335) {};
		\draw node[style=vertex, blue, opacity=0.2](b32) at (0.4499498753135653, 0.4508981386417313) {};
		\draw node[style=vertex, blue, opacity=0.2](b33) at (0.38891852747548, 0.30709231169657303) {};
		\draw node[style=vertex, blue, opacity=0.2](b34) at (0.39865968824833464, 0.9487639026064536) {};
		\draw node[style=vertex, blue, opacity=0.2](b35) at (0.5576460665395266, 0.7032518949272571) {};
		\draw node[style=vertex, blue, opacity=0.2](b36) at (0.9115345587804586, 0.5494549646210898) {};
		\draw node[style=vertex, blue, opacity=0.2](b37) at (0.34682868883052254, 0.9280166269485455) {};
		\draw node[style=vertex, blue, opacity=0.2](b38) at (0.7481131995027119, 0.5881851499551537) {};
		\draw node[style=vertex, blue, opacity=0.2](b39) at (0.4509061568540363, 0.47393014012698004) {};
		\draw node[style=vertex, blue, opacity=0.2](b40) at (0.13529805384033441, 0.4334868562389169) {};
		\draw node[style=vertex, blue, opacity=0.2](b41) at (0.7406483598353373, 0.5899294395266951) {};
		\draw node[style=vertex, blue, opacity=0.2](b42) at (0.7189663407450523, 0.6985779542691091) {};
		\draw node[style=vertex, blue, opacity=0.2](b43) at (0.23455220349342396, 0.8740269690474276) {};
		\draw node[style=vertex, blue, opacity=0.2](b44) at (0.31133765764658794, 0.8591735054747405) {};
		\draw node[style=vertex, blue, opacity=0.2](b45) at (0.7899211102751608, 0.9155502167907655) {};
		\draw node[style=vertex, blue, opacity=0.2](b46) at (0.8304120694344994, 0.9500098266989916) {};
		\draw node[style=vertex, blue, opacity=0.2](b47) at (0.4411245580617038, 0.4356238738196557) {};
		\draw node[style=vertex, blue, opacity=0.2](b48) at (0.5022960474544931, 0.4066208845413454) {};
		\draw node[style=vertex, blue, opacity=0.2](b49) at (0.8690511144906361, 0.6839910135836759) {};
		\draw node[style=vertex, blue, opacity=0.2](b50) at (0.9021329895094164, 0.8666566899148342) {};
		\draw node[style=vertex, blue, opacity=0.2](b51) at (0.9002541284573323, 0.9204076123264435) {};
		\draw node[style=vertex, blue, opacity=0.2](b52) at (0.649013098830954, 0.18150141900310174) {};
		\draw node[style=vertex, blue, opacity=0.2](b53) at (0.0030802352852373266, 0.9295759378548734) {};
		\draw node[style=vertex, blue, opacity=0.2](b54) at (0.13646722073555986, 0.5949479698341059) {};
		\draw node[style=vertex, blue, opacity=0.2](b55) at (0.26159738904931784, 0.09185994970807301) {};
		\draw node[style=vertex, blue, opacity=0.2](b56) at (0.44313491893997353, 0.6548459974155452) {};
		\draw (r0) -- (b17);
		\draw (r1) -- (b10);
		\draw (r2) -- (b12);
		\draw (r3) -- (b19);
		\draw (r4) -- (b5);
		\draw (r5) -- (b7);
		\draw (r6) -- (b8);
		\draw (r7) -- (b11);
		\draw (r8) -- (b3);
		\draw (r9) -- (b2);
		\draw (r10) -- (b9);
		\draw (r11) -- (b15);
		\draw (r12) -- (b20);
		\draw (r13) -- (b4);
		\draw (r14) -- (b21);
		\draw (r15) -- (b18);
		\draw (r16) -- (b23);
		\draw (r17) -- (b6);
		\draw (r18) -- (b16);
		\draw (r19) -- (b1);
		\draw (r20) -- (b24);
		\draw (r21) -- (b13);
		\draw (r22) -- (b0);
		\draw (r23) -- (b14);
		\draw (r24) -- (b22);
	\end{tikzpicture}
	\caption{The square $[0,1]^2$ split into $M \times M$ boxes. The red and blue points have been sampled from a Poisson point process with mean 50, and 25 of each have been selected uniformly at random to form the red and blue permutations respectively. The selected points have been matched using a minimax matching.
	}
\end{figure}

	Let $M = \ceil{n^{2/3} \log^{-1/3}(n)}$ and split the square $[0,1]^2$ into an $M \times M$ grid with evenly spaced rows and columns. Let $S_{i,j}$ denote the square in the $i$th row and $j$th column. That is,
	\[ S_{i,j} = \left[ \frac{j-1}{M}, \frac{j}{M} \right) \times \left[ \frac{i-1}{M}, \frac{i}{M} \right). \]

	Suppose that we have a set of red points such that no two red points are in the same row or column. Then the pattern of the corresponding subpermutation is entirely determined by the squares, and we do not need to know the particular points. If the squares for the matched blue points follow the same pattern, then the red points and blue points correspond to order-isomorphic subpermutations of $\sigma$ and $\pi$.

	To find such sets of points, we will look for matchings in which all edges have the same label in an auxiliary edge-labelled bipartite multigraph $G$ constructed as follows. Let $G$ have vertex set $\{i_1, \dots, i_{M}\} \cup \{j_1, \dots, j_{M}\}$. Given a red point $r$ in row $r_\text{row}$ and column $r_\text{col}$ matched with a blue point $b$ in row $b_\text{row}$ and column $b_\text{col}$, add an edge from row $i_{r_\text{row}}$ to column $j_{r_\text{col}}$ in the graph and label it ($b_\text{row} - r_\text{row}$, $b_\text{col} - r_\text{col}$), i.e. label it with the relative position of $b$ from $r$. Note that a matching in $G$ represents red points in different rows and columns, so a matching in which all edges have the same label, as explained above, corresponds to order-isomorphic subpermutations of $\sigma$ and $\pi$. It will therefore suffice to decompose the edges of $G$ into such matchings. Equivalently, we wish to separately properly edge-colour the edges with each label. Note that as each matched edge has length at most $C \log^{3/4}(n)/\sqrt{n}$, there are at most
	\begin{equation}
		\label{eq:numberlabels}
		\left( 2M \frac{C \log^{3/4}(n)}{\sqrt{n}} + 2\right)^2 = O \left(\frac{M^2 \log^{3/2}(n)}{n}\right)
	\end{equation}
	possible labels. We will handle each label separately.

	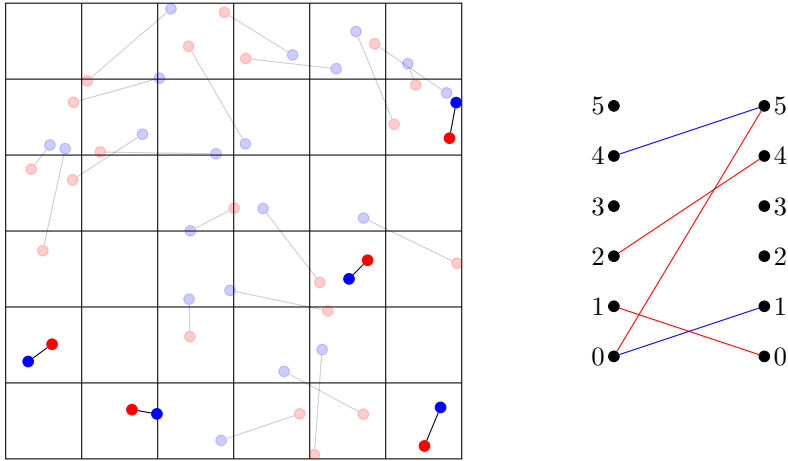
\begin{figure}
	\centering
	\begin{subfigure}{0.5\textwidth}
		\centering
		\begin{tikzpicture}[xscale=\textwidth/1cm, yscale=\textwidth/1cm]
			\draw (0,0) rectangle (1,1);
			\clip (0,0) rectangle (1,1);
			\draw (0, 0.16666666666666666) -- (1,0.16666666666666666);
			\draw (0.16666666666666666, 0) -- (0.16666666666666666,1);
			\draw (0, 0.3333333333333333) -- (1,0.3333333333333333);
			\draw (0.3333333333333333, 0) -- (0.3333333333333333,1);
			\draw (0, 0.5) -- (1,0.5);
			\draw (0.5, 0) -- (0.5,1);
			\draw (0, 0.6666666666666666) -- (1,0.6666666666666666);
			\draw (0.6666666666666666, 0) -- (0.6666666666666666,1);
			\draw (0, 0.8333333333333333) -- (1,0.8333333333333333);
			\draw (0.8333333333333333, 0) -- (0.8333333333333333,1);
			\draw node[style=vertex, red, opacity=0.2](r0) at (0.7062022986281737, 0.324695981345315) {};
			\draw node[style=vertex, red, opacity=1](r1) at (0.9183121824649292, 0.028324099172164646) {};
			\draw node[style=vertex, red, opacity=1](r2) at (0.27684008675089405, 0.10795370718547159) {};
			\draw node[style=vertex, red, opacity=0.2](r3) at (0.688964132031175, 0.3873405687215951) {};
			\draw node[style=vertex, red, opacity=1](r4) at (0.9732261563544754, 0.703766092413252) {};
			\draw node[style=vertex, red, opacity=0.2](r5) at (0.17882999162113236, 0.8293829848866756) {};
			\draw node[style=vertex, red, opacity=0.2](r6) at (0.5009156073330866, 0.5501677971073005) {};
			\draw node[style=vertex, red, opacity=0.2](r7) at (0.8091817311759989, 0.9110325833997828) {};
			\draw node[style=vertex, red, opacity=1](r8) at (0.10169909346112732, 0.25160783763136674) {};
			\draw node[style=vertex, red, opacity=1](r9) at (0.793569013681623, 0.4358683449015855) {};
			\draw node[style=vertex, red, opacity=0.2](r10) at (0.40121237516580055, 0.9053711225579613) {};
			\draw node[style=vertex, red, opacity=0.2](r11) at (0.8991532583697669, 0.8204098821475377) {};
			\draw node[style=vertex, red, opacity=0.2](r12) at (0.8518360621729151, 0.7340607988640941) {};
			\draw node[style=vertex, red, opacity=0.2](r13) at (0.05603357123186831, 0.6356293952095554) {};
			\draw node[style=vertex, red, opacity=0.2](r14) at (0.526397275588272, 0.8784053137975155) {};
			\draw node[style=vertex, red, opacity=0.2](r15) at (0.40401207249635, 0.26842294951046686) {};
			\draw node[style=vertex, red, opacity=0.2](r16) at (0.20755984390116453, 0.6733482699604347) {};
			\draw node[style=vertex, red, opacity=0.2](r17) at (0.7841996937064648, 0.09822530456504389) {};
			\draw node[style=vertex, red, opacity=0.2](r18) at (0.47958401300371645, 0.97988320932866) {};
			\draw node[style=vertex, red, opacity=0.2](r19) at (0.1468185772481496, 0.6122878788851661) {};
			\draw node[style=vertex, red, opacity=0.2](r20) at (0.644669289950765, 0.09893455882012035) {};
			\draw node[style=vertex, red, opacity=0.2](r21) at (0.9894998770087428, 0.4294308306054885) {};
			\draw node[style=vertex, red, opacity=0.2](r22) at (0.6768956574872981, 0.009005649517304318) {};
			\draw node[style=vertex, red, opacity=0.2](r23) at (0.14881834361344962, 0.782510551943364) {};
			\draw node[style=vertex, red, opacity=0.2](r24) at (0.08152244603214623, 0.4566352000150366) {};
			\draw node[style=vertex, blue, opacity=0.2](b0) at (0.6938976037330391, 0.23973900040897317) {};
			\draw node[style=vertex, blue, opacity=0.2](b1) at (0.3000736334880179, 0.7124350115982306) {};
			\draw node[style=vertex, blue, opacity=1](b2) at (0.7528561061327365, 0.394996879900452) {};
			\draw node[style=vertex, blue, opacity=1](b3) at (0.04936821924435675, 0.21369001954937808) {};
			\draw node[style=vertex, blue, opacity=0.2](b4) at (0.09705309579255642, 0.6886834849920827) {};
			\draw node[style=vertex, blue, opacity=1](b5) at (0.98787881984423, 0.7818451025162539) {};
			\draw node[style=vertex, blue, opacity=0.2](b6) at (0.6105225934515158, 0.19194519474705424) {};
			\draw node[style=vertex, blue, opacity=0.2](b7) at (0.36246567950048414, 0.9877394503704601) {};
			\draw node[style=vertex, blue, opacity=0.2](b8) at (0.4048108913656835, 0.5004020850407428) {};
			\draw node[style=vertex, blue, opacity=0.2](b9) at (0.5254765886937326, 0.6913273639945751) {};
			\draw node[style=vertex, blue, opacity=1](b10) at (0.9535537002018528, 0.11295314840676485) {};
			\draw node[style=vertex, blue, opacity=0.2](b11) at (0.9668115930107262, 0.8030551640855773) {};
			\draw node[style=vertex, blue, opacity=1](b12) at (0.33164126184037, 0.09869472771709766) {};
			\draw node[style=vertex, blue, opacity=0.2](b13) at (0.7846824204148839, 0.5284399200131894) {};
			\draw node[style=vertex, blue, opacity=0.2](b14) at (0.33704350716597264, 0.8353973852804982) {};
			\draw node[style=vertex, blue, opacity=0.2](b15) at (0.881749395978703, 0.8671308298783007) {};
			\draw node[style=vertex, blue, opacity=0.2](b16) at (0.6291805780292972, 0.8864767696647602) {};
			\draw node[style=vertex, blue, opacity=0.2](b17) at (0.4917101898423507, 0.36984602622062246) {};
			\draw node[style=vertex, blue, opacity=0.2](b18) at (0.4020456201628176, 0.35047670730054836) {};
			\draw node[style=vertex, blue, opacity=0.2](b19) at (0.5644461012252803, 0.5492333099529227) {};
			\draw node[style=vertex, blue, opacity=0.2](b20) at (0.7679801384914596, 0.9379348089869762) {};
			\draw node[style=vertex, blue, opacity=0.2](b21) at (0.7246061308308218, 0.8562918767366846) {};
			\draw node[style=vertex, blue, opacity=0.2](b22) at (0.13039169443112283, 0.6808427619943374) {};
			\draw node[style=vertex, blue, opacity=0.2](b23) at (0.46125777361941556, 0.6694881357773304) {};
			\draw node[style=vertex, blue, opacity=0.2](b24) at (0.4724058726077113, 0.04043879657051285) {};
			\draw[opacity=0.2] (r0) -- (b17);0.2
			\draw (r1) -- (b10);
			\draw (r2) -- (b12);
			\draw[opacity=0.2] (r3) -- (b19);
			\draw (r4) -- (b5);
			\draw[opacity=0.2] (r5) -- (b7);
			\draw[opacity=0.2] (r6) -- (b8);
			\draw[opacity=0.2] (r7) -- (b11);
			\draw (r8) -- (b3);
			\draw (r9) -- (b2);
			\draw[opacity=0.2] (r10) -- (b9);
			\draw[opacity=0.2] (r11) -- (b15);
			\draw[opacity=0.2] (r12) -- (b20);
			\draw[opacity=0.2] (r13) -- (b4);
			\draw[opacity=0.2] (r14) -- (b21);
			\draw[opacity=0.2] (r15) -- (b18);
			\draw[opacity=0.2] (r16) -- (b23);
			\draw[opacity=0.2] (r17) -- (b6);
			\draw[opacity=0.2] (r18) -- (b16);
			\draw[opacity=0.2] (r19) -- (b1);
			\draw[opacity=0.2] (r20) -- (b24);
			\draw[opacity=0.2] (r21) -- (b13);
			\draw[opacity=0.2] (r22) -- (b0);
			\draw[opacity=0.2] (r23) -- (b14);
			\draw[opacity=0.2] (r24) -- (b22);
		\end{tikzpicture}%
	\end{subfigure}%
	\begin{subfigure}{0.5\textwidth}
		\centering
		\begin{tikzpicture}[yscale=4/6]
			\draw node[vertex](row0) at (0, 0) {};
			\draw node[left] at (0, 0) {0};
			\draw  node[vertex](col0) at (2, 0) {};
			\draw node[right] at (2, 0) {0};
			\draw node[vertex](row1) at (0, 1) {};
			\draw node[left] at (0, 1) {1};
			\draw  node[vertex](col1) at (2, 1) {};
			\draw node[right] at (2, 1) {1};
			\draw node[vertex](row2) at (0, 2) {};
			\draw node[left] at (0, 2) {2};
			\draw  node[vertex](col2) at (2, 2) {};
			\draw node[right] at (2, 2) {2};
			\draw node[vertex](row3) at (0, 3) {};
			\draw node[left] at (0, 3) {3};
			\draw  node[vertex](col3) at (2, 3) {};
			\draw node[right] at (2, 3) {3};
			\draw node[vertex](row4) at (0, 4) {};
			\draw node[left] at (0, 4) {4};
			\draw  node[vertex](col4) at (2, 4) {};
			\draw node[right] at (2, 4) {4};
			\draw node[vertex](row5) at (0, 5) {};
			\draw node[left] at (0, 5) {5};
			\draw  node[vertex](col5) at (2, 5) {};
			\draw node[right] at (2, 5) {5};
			\draw[red] (row0) to (col5);
			\draw[red] (row1) to (col0);
			\draw[red] (row2) to (col4);
			\draw[blue] (row0) to (col1);
			\draw[blue] (row4) to (col5);
		\end{tikzpicture}
	\end{subfigure}
	\caption{A sample minimax matching along with the bipartite graph for the label $(0,0)$. The 5 edges can be coloured using $2$ colours, and so the subpermutations of length 5 can be decomposed into $2$ parts.}
\end{figure}

	It is well-known that a bipartite multigraph can be edge-coloured using $\Delta$ colours where $\Delta$ is the maximum degree of the graph. In fact, this can be done efficiently (see e.g.~Section 20.9 in \cite{schrijver2003combinatorial}). We bound the degree in two steps: first we bound the number of distinct edges labelled $(c_1, c_2)$ that a vertex sees, then we bound the maximum number of repetitions of an edge.

	There is an edge from row $i$ to column $j$ with label $(c_1, c_2)$ if and only if there is a red point in the square $S_{i,j}$ matched to a blue point in the square $S_{ i + c_1, j + c_2}$. In particular, this means that there must be a point from the Poisson process $R$ in $S_{i,j}$ and from the process $B$ in the square $S_{ i + c_1, j + c_2}$. This happens with probability \[\left(1 - \exp\left( - \frac{2n}{M^2} \right)\right)^2 \leq \frac{4n^2}{M^4}.\]

	As the number of points of a Poisson process in a particular square is independent of the numbers in the other squares, the number of distinct edges labelled $(c_1, c_2)$ that are incident with a given vertex in $G$, is dominated by a binomial random variable with $M$ trials and success probability $4n^2/M^4$. Let $X \sim \Bin\left( M, 4n^2/M^4\right)$. Then
	\[\mathbb{E}\left[ X\right] = 4n^2/M^3 \leq \frac{4n^2}{\left(n^{2/3} \log^{-1/3}(n)\right)^3} = 4 \log n.\]
	We also have
	\[4n^2/M^3 \geq \frac{4n^2}{\left(n^{2/3} \log^{-1/3}(n) + 1\right)^3} \geq 3 \log n\]
	for $n \geq 64$.
	Hence, using a well-known Chernoff bound (Theorem 4.4 in \cite{mitzenmacher2017probability}),
	\begin{align*}
		\mathbb{P}\left(X \geq 12 \log n \right) & \leq \mathbb{P} \left( X \geq 12n^2/M^3\right) \\
		                                         & \leq \exp \left(- 4 n^2/M^3\right)             \\
		                                         & \leq \exp \left( - 3 \log n \right)            \\
		                                         & = n^{-3}.
	\end{align*}

	There are $2M$ vertices in $G$ and there are $O \big({M^2 \log^{6/4}(n)}/{n}\big)$ labels (see (\ref{eq:numberlabels})), so the probability that any vertex is incident with at least $12\log n$ distinct edges of the same label is at most
	\begin{align*}
		n^{-3} \cdot 2M \cdot O \left(\frac{M^2 \log^{3/2}(n)}{n}\right)  = O \left( \frac{\log^{1/2}(n)  }{n^{2}} \right)
		= o(n^{-1}).
	\end{align*}

	We now look to bound the number of times a particular edge appears in the graph. For there to be 8 edges from row $i$ to column $j$, the square $S_{i,j}$ must contain at least 8 points of the Poisson process $R$. The number of such points in any given square is a Poisson random variable with mean \[2n/M^2 \leq \frac{2 \log^{2/3}(n)}{n^{1/3}}.\] Hence, using the bound from Lemma \ref{lem:poisson-tail}, the probability that there are at least 8 points in $S_{i,j}$ is at most
	\[\frac{ (2e)^8 \log^{16/3}(n)}{8^8 n^{8/3}} \exp\left( - \frac{2 \log^{2/3}(n)}{n^{1/3}} \right) \leq \frac{\log^{16/3}(n)}{n^{8/3}}. \]
	Taking the union bound over all $M^2$ squares, the probability that any square contains 8 points is at most $O \big({\log^{14/3}(n)}/{n^{4/3}}\big) = o(n^{-1})$.

	Consider the subgraph of $G$ formed by the edges with label $(c_1, c_2)$. With high probability, every edge is repeated at most 7 times and every vertex sees at most $12 \log n$ distinct edges. Hence, the maximum degree is bounded by $84 \log n$, and so we can cover the edges with $84 \log n$ matchings. We can do this separately for every distinct label and cover all edges in at most
	\[O \left(\frac{M^2 \log^{6/4}(n)}{n}\right) \cdot 84 \log n = O\left( n^{{1}/{3}} \log^{11/6}(n) \right)\]
	matchings, each consisting of edges with the same label as required.
\end{proof}

We remark that this proof gives a simple, efficient algorithm to decompose two random permutations $\sigma$ and $\pi$ into twins. More explicitly, given a permutation $\sigma$ of length $n$, generate $n$ red points in $[0,1]^2$ by first sampling $n$ values $x_1, \dots, x_n$ uniformly in $[0,1]$ to be the $x$ coordinates and then sampling $n$ values $y_1, \dots, y_n$ uniformly in $[0,1]$ to be the $y$ coordinates. By relabelling the values as necessary, we can assume that $x_1 < \dotsb < x_n$ and $y_1 < \dotsb < y_n$. Let the $n$ red points in $[0,1]^2$ be $(x_i, y_{\sigma(i)})$, and similarly choose $n$ blue points in $[0,1]^2$ for $\pi$. We then find a minimax matching between the red and blue points. This can easily be done in $O(n^{5/2} \log n)$ time by using a binary search over the maximum allowed edge weight (from the $\binom{n}{2}$ possibilities) and checking if there is a perfect matching using the allowed edges in $O(n^{5/2})$ time \cite{hopcroft1973n}. Crudely, there are at most $n$ possible labels, and for each label we need to edge-colour a bipartite multigraph with at most $n$ edges. Each colouring can be done in $O(n \log n)$ time \cite{alon2003simple,schrijver2003combinatorial}. This process therefore decomposes the permutations into order-isomorphic subpermutations in $O(n^{5/2} \log n)$ time.

The decomposition produced by the above process may not be optimal. However, the proof of Theorem~\ref{thm:joint-twins} shows that the algorithm uses on average $O\big( n^{{1}/{3}} \log^{11/6}(n) \big)$ subpermutations, which is not far from the simple lower bound of $\Omega(n^{1/3})$.

The preceding argument can be generalised with very few modifications to prove Theorem~\ref{thm:multi}. By starting with $k$ independent Poisson processes with rate $2n$, we can generate $k$ collections of $n$ uniformly distributed points in $[0,1]^2$ giving rise to independent uniformly random permutations $\sigma_1,\ldots, \sigma_k \in S_n$. Let the $n$ points determining $\sigma_i$ be $P_i$. The points $P_1$ will play the role of the red points in the proof of Theorem \ref{thm:joint-twins}, while the other sets of points mirror the blue points. Accordingly, we match up the points $P_1$ and $P_i$ for each $i = 2, \dots, k$ using minimax matchings. By Theorem~\ref{thm:leighton-shor}, with probability $1 - o(n^{-1})$, no edge in any of the matchings is longer than $C \log^{3/4}(n)/\sqrt{n}$.

Take $M = \ceil{n^{\frac12+\frac{1}{2(2k-1)}} \log^{-\frac1{2k-1}}(n)}$ and split the square $[0,1]^2$ into an $M \times M$ grid with evenly spaced rows and columns. We construct a similar bipartite multigraph as in the proof of Theorem \ref{thm:joint-twins}. Namely, suppose we have a `red' point $p_1$ matched to points $p_2, \ldots, p_k$, where $p_x\in P_x$ is in row $r_x$ and column $c_x$ for each $x\in [k]$. We then add an edge from row vertex $i_{r_1}$ to column vertex $j_{c_1}$ in the graph and label it by
\[	(r_2-r_1, c_2-c_1, r_3-r_1,c_3-c_1, \ldots, r_k-r_1, c_k-c_1).\]
With probability $1 - o(n^{-1})$, there are at most \[O \left( \frac{M^{2(k-1)}
		\log^{3(k-1)/2}(n)}{n^{k-1}} \right)\]
distinct labels. To bound the degree of the multigraph, it suffices to show that the probability of a vertex being incident to $3 \cdot 2^k \log n$ distinct edges with the same label is $o(n^{-1})$, and the probability that an edge has multiplicity at least $4k$ is also $o(n^{-1})$. This can be done as in the proof of Theorem \ref{thm:joint-twins}. Hence, with probability $1 - o(n^{-1})$, the permutations are $\ell$-similar for
\begin{align*}
	\ell & = O \left( \frac{M^{2(k-1)} \log^{3(k-1)/2}(n)}{n^{k-1}} \right) \cdot  3 \cdot 2^k(4k - 1) \log n \\
	     & = O \left( n^{\frac{k-1}{2k-1}} \log^{\frac{3(k-1)}{2} + \frac{1}{(2k-1)}}(n) \right).
\end{align*}

\section{Open problems}

Theorem \ref{thm:joint-twins} is tight up to the polylogarithmic factor and it would be interesting to know whether some polylogarithmic factor is necessary, or if two uniformly random permutations are $O(n^{1/3})$-similar. Even if a polylogarithmic factor is necessary, it is likely that the $\log^{11/6}(n)$ we have shown can be improved. For example, we have used a minimax matching which guarantees that the number of labels is small and assumed that every label is equally bad, but it is likely that some of the labels have fewer edges and the edges can be covered using fewer matchings. Indeed, if we replace the minimax matching with a matching that minimises the transportation distance, the average distance between a red point and the matched blue point is $\Theta\big(\log^{1/2}(n)/\sqrt{n}\big)$, saving a factor of $\log^{{1}/{4}}(n)$. This suggests that most edges should be spread over fewer labels and that the polylogarithmic factor may be improved upon.

\begin{problem}
Are two permutations of length $n$ chosen independently and uniformly at random $O(n^{1/3})$-similar with high probability?
\end{problem}

The other main problem pertaining to (general) twins in permutations is to determine the maximum size of a twin that one is guaranteed to find in a permutation of length $n$. The best lower bound, due to Bukh and Rudenko~\cite{bukh2020order}, is currently $\Omega(n^{3/5})$, while the best upper bound is $O(n^{{2}/{3}})$ obtained by applying a first moment calculation to a random permutation \cite{gawron2014izomorficzne}. It would be interesting to narrow this gap. Resolving whether the upper bound is tight, as conjectured by Gawron \cite{gawron2014izomorficzne}, would be of particular interest as this would determine whether there are permutations for which the maximum twin length is asymptotically smaller than that of a random permutation.

\begin{problem}[Gawron \cite{gawron2014izomorficzne}]
Does every permutation of length $n$ contain a twin of length $\Omega(n^{2/3})$?
\end{problem}

\paragraph{Acknowledgements}
We would like to thank the anonymous referees for their helpful comments.

\bibliographystyle{abbrev-bold}
\bibliography{twins}
\end{document}